\documentclass[onefignum,onetabnum]{siamart171218}

\usepackage{amsfonts}
\usepackage{graphicx}
\usepackage{caption}
\usepackage{subcaption}
\usepackage{mathtools}
\usepackage{amssymb}
\usepackage{multirow}%%%%%%%%%%
\usepackage{amsmath}%%%%%%%
\usepackage{algorithm}%%%%%%%%%%%%%
\usepackage{algpseudocode}%%%%%%%%%%%%
\usepackage{hyperref}%%%%%%%%%%%
\usepackage{xcolor}
\definecolor{darkgreen}{cmyk}{0.7, 0, 1, 0.2}
\usepackage{array}
\usepackage{mathabx}
\usepackage{bm}
\usepackage[noadjust]{cite} 
\usepackage{comment}
\usepackage{amsopn}
\usepackage{enumitem}
\usepackage{placeins}
\usepackage{csquotes}
\usepackage{float}%%%%%%%

\ifpdf
\DeclareGraphicsExtensions{.eps,.pdf,.eps,.jpg}
\else
\DeclareGraphicsExtensions{.eps}
\fi

\newtheorem{Remark}[theorem]{Remark}
\headers{Stable Block Solutions for Ill-conditioned Linear Systems}{Suvendu Kar and Murugesan Venkatapathi}

\title{A fast solver for ill-conditioned linear systems using randomized stable solutions of its blocks}

\author{Suvendu Kar \and Murugesan Venkatapathi\thanks{Department of Computational \& Data Sciences, Indian Institute of Science, Bangalore. (\email{suvendukar@iisc.ac.in} \and \email{murugesh@iisc.ac.in})}}

\makeatletter
\newcommand*{\addFileDependency}[1]{% argument=file name and extension
	\typeout{(#1)}% latexmk will find this if $recorder=0 (however, in that case, it will ignore #1 if it is a .aux or .pdf file etc and it exists! if it doesn't exist, it will appear in the list of dependents regardless)
	\@addtofilelist{#1}% if you want it to appear in \listfiles, not really necessary and latexmk doesn't use this
	\IfFileExists{#1}{}{\typeout{No file #1.}}% latexmk will find this message if #1 doesn't exist (yet)
}
\makeatother

\DeclarePairedDelimiter\norm{\lVert}{\rVert}
\begin{document}

\maketitle
	
\begin{abstract}
We present an enhanced version of the row-based randomized block-Kaczmarz method to solve a linear system of equations. This improvement makes use of a regularization during block updates in the solution, and a dynamic proposal distribution based on the current residual vector and effective mutual orthogonality between all blocks. The improved method provides significant gains in solving highly ill-conditioned linear systems that are either sparse, or dense least-squares problems that are significantly over/under determined. Considering the poor guarantees in effectively preconditioning iterative solutions for such ill-conditioned problems, it may also serve as a pre-solver for accelerating other iterative numerical methods, and as an inner iteration in certain types of GMRES solvers for linear systems.
\end{abstract}
	
	% REQUIRED
	
	% REQUIRED
%%%%%%%%%%%%%%%%%%%%%%%%%%%%%%%%%%%%%%%%%%%%%%%

\section{Introduction}
In solving a consistent linear system of equations
\begin{equation}
\label{eq:linear_system}
    Ax=b
\end{equation}
where $A \in \mathbb{R}^{m \times n}, \text{and} \quad b\in \mathbb{R}^{m\times1} $ are given, Kaczmarz-type algorithms solve for $x \in \mathbb{R}^{n\times1}$ sequentially enforcing one equation of the system in every iteration \cite{1,2,3,4,5,6,7,8}. This procedure dates back to its original formulation in 1937 \cite{1} and it later resurfaced in tomographic reconstruction as the Algebraic Reconstruction Technique (ART) \cite{9}. The randomized Kaczmarz method was later introduced to enhance its performance, achieving an \textit{expected} exponential convergence rate \cite{2}. A greedy randomized Kaczmarz scheme was also developed, delivering significantly faster convergence for obtaining approximate solutions to \autoref{eq:linear_system} \cite{3}.

A flexible GMRES solver preconditioned by Kaczmarz‐type inner iterations was recently proposed for \autoref{eq:linear_system} \cite{10}. This FAB‐GMRES framework \cite{10,11,12}, uses inner Kaczmarz iterations as a preconditioner and can deliver markedly faster convergence than manually designed preconditioners. To avoid the evaluation of $AA^T$ in the above, the use of block solutions was introduced in a preconditioned orthogonal block-Kaczmarz routine for the same \cite{14}. The Reverse Cuthill–McKee method \cite{13} played a significant role in the formulation of this routine. The approach utilizes a block partitioning criterion based on the cosine of angles between two block matrices and operates through two orthogonal projections onto nearly orthogonal hyperplanes in every iteration, ensuring that successive error reduction occurs in complementary directions. Thus it converges much more rapidly than the prior block based methods, both in theory and practice, while it is limited to square systems. To overcome the restriction to square linear systems, a Simple Orthogonal Block-Kaczmarz (SOBK) method was introduced and subsequently SOBK was embedded as an inner preconditioner within flexible GMRES to address ill-conditioned problems \cite{0}.

We propose enhancements in the block-Kaczmarz method by 1) including residual-based dynamic aggregation into a block in each iteration, 2) generalizing the idea of mutual orthogonality between two blocks to an effective orthogonality of a given block with \emph{all} other blocks, and 3) by incorporating a regularization in the iterations, rendering it more stable \cite{ReRBK}. This effective utilization of regularization and orthogonality of blocks significantly enhances the rate of convergence in our proposed method, referred here as the Regularized Orthogonality and Residual based Block-Kaczmarz (ROR-BK) method.

The proposed method offers a promising alternative for solving linear systems 
$Ax=b$, particularly when the matrix 
$A$ is ill-conditioned. The approach achieves efficient approximations of the solution without the need for explicit preconditioning. Although preconditioners are designed to reduce the condition-number of the matrix, this does not necessarily reduce the condition-number of the problem\footnotemark[1], which is critical for fast convergence \cite{ErrorEstimator}. Preconditioner implementations are subject to numerical instability in practice, particularly for large condition numbers, as they require solving auxiliary systems or applying approximate inverses that can accumulate round-off errors. These limitations become more pronounced with larger or highly ill-conditioned matrices, where preconditioners may experience pivot breakdown, require excessive fill-in for stability, or lose their conditioning properties, often becoming ineffective precisely when robust acceleration is most needed \cite{shortcomingofpreconditioner1,shortcomingofpreconditioner2,shortcomingofpreconditioner3}. Furthermore, many standard preconditioning techniques may fail with non-symmetric and indefinite matrices \cite{LimitationOfPreconditioner}. Also, note that a minimization of the norm of the residual in such highly ill-conditioned problems may not result in a corresponding convergence in the relative error.

\footnotetext[1]{When solving for $x$ given $A,b$ in a linear system of equations $Ax=b$, the condition number of the problem is defined as $\frac{\|A^{\dagger}\|\|b\|}{\|x_\star\|}$. In the results of this paper, the average condition number of the problems solved are reported as 'prob-cond' for any given matrix in tables 2-7. Similarly, the condition number of the problem in evaluating $b$ given $A,x$ is $\frac{\|A\|\|x_\star\|}{\|b\|}$. Any of these condition numbers may strongly affect convergence depending on the iterative method used.}

In contrast, the partitioning of the given linear system into smaller blocks of equations, as in the block-Kaczmarz method, may distribute the highly covariant linear equations among the different blocks rendering each of them potentially well conditioned. This gainful distribution into reasonably well-conditioned blocks is typically done using a trivial aggregation of contiguous rows; note that other methods of optimal aggregation of rows into blocks with reordering based on mutual orthogonality may incur an unviable $\mathcal{O}(m^2n)$ arithmetic effort. These block equations typically solved using a Moore-Penrose pseudo-inverse can nevertheless be poorly conditioned in some cases. Alternately, one may also solve a subset of equations that have the maximum residuals in the current iteration, using a dynamic aggregation of these equations into a block, thus improving convergence. But, these dynamic blocks can also yield ill-conditioned least-square problems.

In our work, we use both the residual-based dynamic blocks, and the fixed blocks of contiguous equations that are randomly sampled based on their effective orthogonality with all other blocks. Importantly, we incorporate regularization in all the above block solutions. Updating solutions preferably using the more orthogonal blocks provides us an additional layer of stability in the solutions over the regularization. This effective amalgamation of these concepts in a simple new form provides fast and stable convergence for high-condition number linear systems in general, without the requirement of a pre-conditioner. The proposed ROR-BK method is also beneficial for an approximate weighted least-squares solution of a system of equations $Ax=b$ (see \autoref{thm:dobk_for_least_square_problems} in the appendix for more details).

The remainder of this paper is organized as follows. In \autoref{sec:methods}, we introduce the proposed ROR-BK algorithm and provide a theoretical convergence analysis of the proposed improvements. In  \autoref{sec:numerical_experiments}, we report a series of numerical experiments comparing the proposed method to several existing randomized methods including the SOBK method \cite{0}, and the other Krylov subspace iterative solvers including LSQR and versions of GMRES to illustrate its effectiveness and robustness. Finally, \autoref{sec:conclusion} offers a summary and concluding remarks on the work. In the appendix, we present a few relevant algorithms including the evaluation of an optimal initial solution for iterative solvers.

\subsection{Notation}
The symbol $A^{\dagger}$ denotes the pseudo-inverse of matrix $A \in \mathbb{R}^{m \times n}$. The notation \( A_{(i)} \) and \( A^{(i)} \) represents the \( i^{th} \) row and column of the matrix \( A \), respectively. \(\langle .\) \(, . \rangle\) represents an inner product defined by \(\langle v_1\) \(, v_2 \rangle\)$=v_1^Tv_2$. Unless specifically mentioned, $\|.\|$ as well as $\|.\|_2$  denotes 2-norm and $\|.\|_F$ denotes Frobenius norm. $x_\star$ denotes the solution minimizing the norm of the residual of \autoref{eq:linear_system}, unless specifically mentioned.

\section{Methods}\label{sec:methods}
First, we recall the Kaczmarz method of solving a linear system of equations. This method uses a single row numbered $i_k$ in its update at the $k^{th}$ iteration:
\begin{align}
x_{k+1} &= x_{k} + \frac{b_{i_k} - \langle A_{(i_k)}, x_{k} \rangle}{\|A_{(i_k)}\|_2^2} A_{(i_k)} \label{eq:update_rule_kaczmarz}
\end{align}
solving the following minimization problem:
\begin{align*}
    x_{k+1}=\text{arg min}_x\|x-x_{k}\|_2^2 \text{  with } \langle A_{(i_k)},x\rangle=b_{(i_k)}.
\end{align*}

This results in an orthogonal projection onto the space defined by the constraint \( \{ x : \langle A_{(i_k)}, x \rangle = b_{i_k} \} \). Equivalently, it exactly solves a given single equation, with a minimum perturbation in a least-squares sense to the current solution $x_k$. In contrast, the block Kaczmarz method employs a block of rows (denoted as \( B \)) for the solution update, given by:
\begin{align}
    x_{k+1} & = x_{k} + B^\dagger (b_{B}-B x_{k})\label{eq:update_rule_block_kaczmarz}
\end{align}
solving, 
\begin{align*}
    x_{k+1}=\text{arg min}_x\|x-x_{k}\|_2^2 \text{  with } Bx=b_{B}.
\end{align*}

Here, \( b_{B} \) represents the sub-vector of \( b \) corresponding to the row indices used to form the block \( B \) from \( A \). Thus, the block method solves a set of multiple equations for a given full rank $B$ and results in an orthogonal projection onto the space defined by the constraints \( \{ x : B x = b_{B} \} \), which corresponds to the intersection of multiple constraint hyperplanes (see \autoref{fig:graph_block_kaczmarz}). \autoref{eq:update_rule_block_kaczmarz} also indicates that the additive update to the solution vector $x_k$ solves a (typically under determined) linear least-square problem of a block of equations and its residual using the pseudo-inverse $B^\dagger$. This ensures that the 2-norm perturbation to the iterate $x_k$ is minimal while also satisfying the new constraints of the current iteration. 

\begin{figure}
    \centering   \includegraphics[width=7cm, height=7cm]{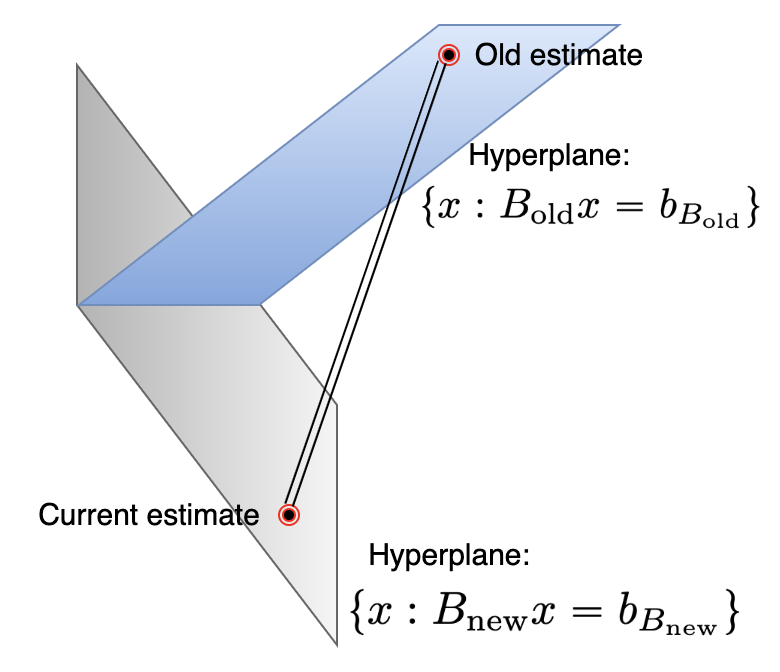}
    \caption{Representation of solution updates in block-Kaczmarz methods.}
    \label{fig:graph_block_kaczmarz}
\end{figure}

Note that this offers the (block) Kaczmarz method a significant advantage over residual-minimizing GMRES in highly ill-conditioned systems. GMRES directly minimizes the residue $\|r_k\|_2 = \|b - Ax_k\|_2$, which primarily attenuates error components corresponding to large singular values, leaving other directions largely unresolved with a poor convergence in the actual \textit{errors}. Also, when exploiting nearly orthogonal block subspaces in the former, successive block updates act in complementary directions, reducing oscillatory behavior in the errors and ensuring more uniform error contraction across directions. Consequently, the block Kaczmarz method can converge to substantially lower relative errors in the solution than GMRES or its variants, providing more robust iterative solutions for ill-conditioned linear systems.

Rest of this section describes the proposed ROR-BK method and its use in the preconditioned flexible GMRES method for solving \autoref{eq:linear_system}. As presented in \cite{10}, an efficient Kaczmarz-type routine as the inner-iteration can be used as the preconditioner for
FAB-GMRES. Performing continuous updates with mutually orthogonal blocks can lead to a faster approximation of the exact solution of \autoref{eq:linear_system} without requiring prior computation of $AA^T$ for the inner iterations in FAB-GMRES. Based on this motivation, Simple Orthogonal Block-Kaczmarz (SOBK) method was proposed\cite{0}. While SOBK samples pairs of mutually orthogonal blocks of contiguous rows, we propose to use the cumulative effective orthogonality of a block with \emph{all} other blocks in its random sampling, along with a dynamically aggregated block based on the current residual, for updating the iterative solution. 
 
We show through \autoref{thm:convergence_boost} that this addition enhances the convergence of the earlier proposed SOBK method. Also, as opposed to the use of $B^\dagger=B^T(BB^T)^{-1}$ directly in \autoref{eq:update_rule_block_kaczmarz} for the update in block-Kaczmarz methods, we use its regularized form $B^T(BB^T+\mu I)^{-1}$ to ensure stability. This choice of the positive constant $\mu$ corresponds to a stochastic proximal point algorithm with a large constant step size $\frac{1}{\mu}$ as discussed in \cite{ReRBK}.

First, a sequential aggregation is used to create blocks for \(A\) and \(b\) as $[A_1,\cdots,A_k]$ and $[b_1,\cdots,b_k]$, respectively, each with at least $\left\lfloor \frac{m}{k} \right\rfloor$ rows in it. The symmetric matrix \(C\) is then evaluated, where the \((i,j)^{th}\) entry of $C$ is defined as $\frac{|\langle\bar{A_i}, \bar{A_j}\rangle|}{\|\bar{A_i}\|\|\bar{A_j}\|}$, indicating the cosine of the effective angle between blocks \(A_i\) and \(A_j\) using corresponding centroids $\bar{A_i}$ and $\bar{A_j}$. The centroid vector is just a sum of all row vectors of the given block of $A$. Since the number of blocks $k\ll m$ the number of rows, the $\mathcal{O}(mn+k^2n)$ evaluation of C(i,j) representing the degree of orthogonality between all pairs of blocks $i,j$ is not cumbersome. To sample a block based on its effective orthogonality with all other blocks, the probability $\mathcal{P}_t$ of a block numbered $t$ is given by:
\begin{equation}
\mathcal{P}_t =\frac{\mathcal{\bar{P}}_t}{\sum_{t=1}^k\mathcal{\bar{P}}_t} \text{ where }\mathcal{\bar{P}}_t = e^{-\frac{2\sum_{j=1}^k C(t,j)}{n}}.
\end{equation}
Note that sampling such a distribution for any one of the $k$ indices of the blocks using rejection sampling is only $\mathcal{O}(k)$ in computing effort in every iteration, along with a one-time $\mathcal{O}(k^2)$ cost in generating the distribution. The number of such updates $l$ of the iterative solution $x_j$ based on orthogonality, for every iteration using a residual based dynamic aggregation of a block, can be determined by the user based on stability requirements. In this work we use $l=3$ as presented in \autoref{algo:roar-bk}.

\begin{algorithm}
\caption{ROR-BK (Regularized Orthogonality and Residual based Block Kaczmarz)}\label{algo:roar-bk}
\begin{algorithmic}[1]
\Require $A, b, x_0, \mu,k$. \Comment{ $k$: number of blocks, and $\mu$: regularization parameter.}
\Ensure $\tilde{x}$\Comment{Approximate solution}
\State Perform uniform contiguous aggregation of rows of $A$ and $b$ in sequential order, and obtain $k$ blocks $A_i$ and $b_i$, $i = 1,2,\ldots,k$. $\bar{A} = [\bar{A}_1, \bar{A}_2, \ldots, \bar{A}_k]$, where $\bar{A}_i$ represents the centroid of $A_i$.
\State Compute the absolute cosine value $C(i,j)$ of the effective angle between blocks $A_i$ and $A_j$ by the centroid coordinates, and the corresponding symmetric matrix $C$.
\State Following $C$, create a list $\mathcal{\bar{P}}$ of length $k$ with it's $t^{\text{th}}$ element defined as $e^{-\frac{2\sum_{j=1}^k C(t,j)}{n}}$.
\State Define probability distribution $\mathcal{P}=\frac{\mathcal{\bar{P}}}{\sum_{t=1}^k\mathcal{\bar{P}}[t]}$
\For{$j =0,1,2,\ldots$ until convergence}
    \For{$l = 1,2,3$}
        \State Sample an index $\tau$ w.r.t the probability distribution $\mathcal{P}$.
        \State $x_{j} \leftarrow x_j + A_{\tau}^T(A_{\tau}A_{\tau}^T+\mu I)^{-1} (b_{\tau} - A_{\tau} x_j)$
    \EndFor
    % \State \label{step:roar-bk}$r = b - A x_{j}$ and pick indices $\mathcal{I}$ of top $\left\lfloor \frac{m}{k} \right\rfloor$ squared elements $r_i^2$ where $i=1,2,\dots m$
    \State $r = b - A x_{j}$
\If{Convergence criterion met}
    \State terminate
\Else
    \State\label{step:roar-bk} pick set of indices $\mathcal{I}$ with top $\left\lfloor \frac{m}{k} \right\rfloor$ squared elements $r_i^2$ where $i=1,2,\dots,m$
\EndIf
    \State $A_{\mathcal{I}}=A(\mathcal{I},:)$ and $b_{\mathcal{I}}=b(\mathcal{I})$
    \State $x_{j+1} \leftarrow x_{j}+A_{\mathcal{I}}^T(A_{\mathcal{I}}A_{\mathcal{I}}^T+\mu I)^{-1}(b_{\mathcal{I}}-A_{\mathcal{I}}x_{j})$
\EndFor
\end{algorithmic}
\end{algorithm}

For dynamically constructing a block based on the residual at each iteration, we gather a fixed number of rows that contribute the most to the current residual. In step:\autoref{step:roar-bk} of \autoref{algo:roar-bk} for a real linear system, we use the sorted order of squared elements of the residual vector $r$ and construct the set of indices $\mathcal{I}$, representing the required $\lfloor \frac{m}{k}\rfloor$ rows, to minimize the 2-norm of the current residual $\|b-Ax_{j}\|_2^2$ in each update of the solution.

The proposed method applies to the different types of linear systems, namely square ($m=n$), overdetermined ($m > n$), and underdetermined ($m < n$) cases. We analyse the convergence properties of Algorithm \autoref{algo:roar-bk}, and also we show that it can further decrease the residual encountered in SOBK, thereby accelerating convergence and improving the decay of relative residual norm (RRN) given by $\frac{\|b-Ax_t\|_2}{\|b\|_2}$, where $x_t$ is the approximated solution at $t^{\text{th}}$ iteration in the proposed ROR-BK method.

\begin{lemma}\label{thm:woodberry_identity}
 For a matrix M $\in \mathbb{R}^{r\times n}$, let $I_r$ denote the $r\times r$ identity matrix, $I_n$ denote the $n \times n$ identity matrix, and $\mu$ be a positive constant, then:
 \begin{equation}\label{eq:woodberry_identity}
 M^T(MM^T+\mu I_r)^{-1}=(M^TM+\mu I_n)^{-1}M^T
 \end{equation}
\end{lemma}

\begin{proof}
Using Woodberry Identity \cite{woodberryidentity} we can simplify $(MM^T+\mu I_r)^{-1}$ and $(M^TM+\mu I_n)^{-1}$ as:
    
\begin{align}
(\mu I_r+MM^T)^{-1} &=(\mu I_r+MI_nM^T)^{-1}\notag\\
&=(\mu I_r)^{-1}-(\mu I_r)^{-1}M\{M^T (\mu I_r)^{-1}M+(I_n)^{-1} \}M^T(\mu I_r)^{-1}\notag \\
&=\frac{1}{\mu}I_r-\frac{1}{\mu^3}M\{M^TM+\mu I_n\}M^T\label{eq:wi1}
\end{align}
and,
\begin{align}
(M^TM+\mu I_n)^{-1}&=(\mu I_n+M^TI_rM)^{-1}\notag\\
&=(\mu I_n)^{-1}-(\mu I_n)^{-1}M^T\{M(\mu I_n)^{-1}M^T+(I_r)^{-1}\}M(\mu I_n)^{-1}\notag\\
&=\frac{1}{\mu}I_n-\frac{1}{\mu^3}M^T\{MM^T+\mu I_r\}M\label{eq:wi2}
\end{align}
Thus, LHS of \autoref{eq:woodberry_identity} becomes:
\begin{align}
M^T(MM^T+\mu I_r)^{-1}&=M^T[\frac{1}{\mu}I_r-\frac{1}{\mu^3}M\{M^TM+\mu I_n\}M^T], \text{using \autoref{eq:wi1}}\notag\\
&=\frac{1}{\mu}M^T-\frac{1}{\mu^3}M^TM\{M^TM+\mu I_n\}M^T\label{eq:wi3}
\end{align}
and RHS of \autoref{eq:woodberry_identity} becomes:
\begin{align}
(M^TM+\mu I_n)^{-1}M^T&=[\frac{1}{\mu}I_n-\frac{1}{\mu^3}M^T\{MM^T+\mu I_r\}M]M^T\text{using \autoref{eq:wi2}}\notag\\
&=\frac{1}{\mu}M^T-\frac{1}{\mu^3}M^T\{MM^T+\mu I_r\}MM^T\notag\\
&=\frac{1}{\mu}M^T-\frac{1}{\mu^3}M^TM\{M^TM+\mu I_n\}M^T\label{eq:wi4}
\end{align}
As R.H.S of \autoref{eq:wi3}, and \autoref{eq:wi4} are identical, the lemma holds. Note that if $r<n$, then the evaluation of LHS in \autoref{eq:woodberry_identity} is cheaper than that of RHS, while the converse is true if $n<r$.
\end{proof}

\begin{lemma}\label{thm:x_in_range_A_transpose}
If $x_0 \in \mathcal{R}(A^T)$, then the approximate solution $\tilde{x}$ generated using Algorithm \autoref{algo:roar-bk} also satisfies $\tilde{x} \in \mathcal{R}(A^T)$. $\mathcal{R}(T)$ denotes the range space of matrix $T$.
\end{lemma}

\begin{proof}
    The update rule for $x_{t+1}$ from $x_{t}$ is as follows:
    \begin{equation}\label{eq:update_rule}
    x_{t+1}=x_{t}+A_{\tau}^T(A_{\tau}A_{\tau}^T+\mu  I)^{-1}(b_{\tau}-A_{\tau}x_{t})
    \end{equation}
    for a chosen block $A_{\tau}$. It is clear from \autoref{eq:update_rule} that if $x_t \in \mathcal{R}(A^T)$, then as $A_{\tau}^T(A_{\tau}A_{\tau}^T+\mu I)^{-1}(b_{\tau}-A_{\tau}x_{t}) \in \mathcal{R}(A^T)$, $x_{t+1}$ is also in $\mathcal{R}(A^T)$. Now using the principle of mathematical induction, the theorem can be proved for any $t$. Hence, the obtained approximate solution $\tilde{x} \in \mathcal{R}(A^T).$ 
\end{proof}

One can show that the above method of choosing blocks can also solve a weighted least square solution of the blocks with the required convergence properties, as presented in the appendix. Note that the proliferation of large-scale datasets in machine learning and scientific computing necessitates the development of algorithms that process only small data subsets per iteration. For linear least-squares problems, the randomized block-Kaczmarz (RBK) method exemplifies such an approach; however, existing convergence guarantees require sampling distributions that may involve computationally prohibitive preprocessing costs. This limitation can be overcome through a randomized block-Kacmarz method with uniform sampling, establishing (as shown in \autoref{thm:dobk_for_least_square_problems}) that the iterates converge in the mean to a weighted least-squares solution. However, the resulting weight matrix can exhibit arbitrarily large condition numbers, and the iterate variance may grow unbounded in the absence of a stable solver like ROR-BK. We now proceed to analyse this proposed method.

\begin{theorem}\label{thm:err_bound}
    Let $x_{t+1}$ and $x_t$ be the $(t+1)^{th}$ and $t^{th}$ updates of the solution in Algorithm \autoref{algo:roar-bk} respectively with respect to a chosen block $A_\tau$. Let $x_\star$ be $A^\dagger b$. Then, 
    \[
    \|x_{t+1} - x_\star\|_2 \leq 
    \frac{\mu}{\lambda_{\min}^{+}(A_\tau^T A_\tau) + \mu} \|x_t - x_\star\|_2
    \]
Here, $\lambda_{\min}^{+}(A_\tau^T A_\tau)$ is the smallest non-zero eigenvalue of $A_\tau^T A_\tau$. Equivalently, \begin{align}
\mathbb{E}[\|x_{t}-x_\star\|_2]\le \left( \frac{\mu}{\lambda_{\text{min}}^{\text{block}}+\mu} \right)^t\mathbb{E}[\|x_0-x_\star\|_2] \notag
\end{align}
where, $\lambda_{\text{min}}^{\text{block}}$ is the minimum non-zero eigenvalue of $A_\tau^TA_\tau$ for all $\tau$.
\end{theorem}
\begin{proof}
\autoref{algo:roar-bk} uses blocks $A_\tau$ in updates as:
\begin{align}
x_{t+1}&=x_{t}+A_\tau^T(A_\tau A_\tau^T+\mu I)^{-1}(b_{\tau}-A_{\tau}x_t)\notag
\end{align}
Thus we have,
\begin{align}
x_{t+1} - x_\star &= x_t - x_\star+ A_\tau^T(A_\tau A_\tau^T+\mu I)^{-1}(b_{\tau}-A_{\tau}x_t) \notag\\
&= x_t - x_\star - A_\tau^T(A_\tau A_\tau^T+\mu I)^{-1}A_{\tau}(x_t-x_\star) \label{eq:errbound0}\\
&= (I - M)(x_t - x_\star)\label{eq:errbound1}
\end{align}
Now using \autoref{thm:woodberry_identity}, $M$ in the above can be rewritten as:
\begin{align}
    M&=A_\tau^T(A_\tau A_\tau^T+\mu I)^{-1}A_{\tau}\notag\\
    &=(A_\tau^TA_\tau+\mu I)^{-1}A_{\tau}^TA_\tau\notag\\
    &=I-\mu(A_\tau^TA_\tau+\mu I)^{-1}\label{eq:errbound2}
\end{align}
Thus, \autoref{eq:errbound1} reduces to,
\begin{align}
\|x_{t+1} - x_\star\|_2 &= \|\mu(A_\tau^TA_\tau+\mu I)^{-1}(x_t - x_\star)\|_2 [\text{Using \autoref{eq:errbound2}}]\notag\\
&\le \frac{\mu}{\lambda_{\min}^{+}(A_\tau^TA_\tau)+\mu}\|x_t-x_\star\|_2\label{ew:errbound3}
\end{align}
So, in expectation:
\begin{align}
\mathbb{E}[\|x_{t}-x_\star\|_2] &\le \mathbb{E}[\frac{\mu}{\lambda_{\min}^{+}(A_\tau^TA_\tau)+\mu}\|x_{t-1}-x_\star\|_2], \text{using \autoref{ew:errbound3}}\notag\\
&\le \left( \frac{\mu}{\lambda_{\text{min}}^{\text{block}}+\mu} \right)\mathbb{E}[\|x_{t-1}-x_\star\|_2]\notag\\
&\le ......\notag\\
&\le \left( \frac{\mu}{\lambda_{\text{min}}^{\text{block}}+\mu} \right)^{t}\mathbb{E}[\|x_{0}-x_\star\|_2]
\end{align}
where, $\lambda_{\text{min}}^{\text{block}}$ is the minimum non-zero eigenvalue of $A_\tau^TA_\tau$ for all $\tau$.
\end{proof}
The above theorem establishes that a relatively small non-zero regularization parameter $\mu$ ensures stable convergence even when any of the blocks in the given system of equations is numerically rank deficient. Though a relatively smaller $\mu$ provides the expected reduction of error as outlined in \autoref{thm:err_bound}, a very small $\mu$ can be ineffective. We use $\mu=1e-6 \times$ \text(number of rows in a block) in our experiments that are described in the later sections. We now show that using a residual based dynamic block of equations for updating the iterative solution can also help to obtain faster convergence in general.

\begin{theorem}\label{thm:convergence_boost}
    In Algorithm \autoref{algo:roar-bk}, when $\mu$ is sufficiently small, updating the solution $x_t$ with respect to a block containing $\left\lfloor\frac{m}{k}\right\rfloor$ rows that contribute the most to the residual, reduces the upper and lower bounds of the error thus improving convergence in general.
\end{theorem}
\begin{proof}
When $\mu$ is sufficiently small, the update rule in \autoref{algo:roar-bk} reduces to :
\begin{align}
    x_{t+1}&\approx x_t+A_{\tau}^\dagger(b_\tau-A_\tau x_t)\notag
\end{align}
Thus, \[
x_{t+1}-x_\star\approx (I-U)(x_{t}-x_\star)
\]
where, $U=A_\tau^\dagger A_\tau$. We know that $U^T=U$ and $U^2=U$. Thus, U is an orthogonal projection matrix, and so is $(I-U)$. Thus,
\begin{align}
    \|x_{t+1}-x_\star\|_2^2&\approx\|(I-U)(x_t-x_\star)\|_2^2\notag\\
    &= (x_t-x_\star)^T(I-U)^T(I-U)(x_t-x_\star)\notag\\
    &= (x_t-x_\star)^T(I-U)(x_t-x_\star)\notag\\
    &= \|x_t-x_\star\|_2^2-(x_t-x_\star)^TU(x_t-x_\star)\notag\\
    &= \|x_t-x_\star\|_2^{2}-(x_t-x_\star)^TU^TU(x_t-x_\star)\notag\\
    &= \|x_t-x_\star\|_2^2-\|U(x_t-x_\star)\|_2^2\notag\\
    &= \|x_t-x_\star\|_2^2-\|A_\tau^\dagger(A_\tau x_t-b)\|_2^2\label{eq:errbound4}
\end{align}

From \autoref{eq:errbound4},
    \begin{align}
      \|x_t-x_\star\|_2^2-\frac{\|A_\tau x_t-b\|_2^{2}}{\lambda_{\min}^{+}(A_\tau^{T}A_\tau)}\lesssim \|x_{t+1}-x_\star\|_2^2 & \lesssim \|x_t-x_\star\|_2^2-\frac{\|A_\tau x_t-b\|_2^{2}}{\lambda_{\max}^{+}(A_\tau^{T}A_\tau)}\label{eq:errbound5}
    \end{align}
where, $\lambda_{\min}^{+}(A_\tau^T A_\tau), \lambda_{\max}^{+}(A_\tau^TA_\tau)$ are the smallest and largest non-zero eigenvalues of $A_\tau^T A_\tau$ respectively. Now it is evident from \autoref{eq:errbound5} that given a $x_t$, if we choose the set of the indices $S$ of rows such that $\|(Ax_t)_S-b_S\|_2^2\ge\|(Ax_t)_{S'}-b_{S'}\|_2^2$ for any other set of indices $S'$, with $|S|=|S'|=\left\lfloor\frac{m}{k}\right\rfloor$, we are likely to minimize the above upper and lower bounds of the error $\|x_{t+1}-x_\star\|_2^2$ in \autoref{eq:errbound5} in general, as $\lambda_{\min}^{+}, \lambda_{\min}^{-}$ are bounded for a given system, and thus ensure faster convergence over many iterations.
\end{proof}

We utilize the proposed ROR-BK approach as an inner iteration to construct a flexible AB-GMRES algorithm, where the ROR-BK method serves as the preconditioner. This strategy is particularly advantageous for solving large-scale linear systems that are ill-conditioned. Refer to \autoref{algo:fabgmres_dobk} presented in the \autoref{sec:appendix}  for more details.

\begin{Remark}
If $z_0$ (the initial solution with which ROR-BK as an inner-iteration starts) and $x_0 \in \mathcal{R}(A^{T})$, FAB-GMRES with an inner iteration of ROR-BK (\autoref{algo:fabgmres_dobk}) gives a solution $x_k \in \mathcal{R}(A^T)$ minimizing $\norm{b-Ax_k}_2$.
\end{Remark}

\begin{proof}
    With $z_0 \in \mathcal{R}(A^T)$ as an initial solution for the inner iteration step, ROR-BK in \autoref{algo:fabgmres_dobk} guarantees that the approximate solution for $Az=v_k$ $\in \mathcal{R}(A^T)$ by \autoref{thm:x_in_range_A_transpose}. Thus, $z_{k} \in \mathcal{R}(A^T)$. Now from the last step of \autoref{algo:fabgmres_dobk}, it is clear that when $z_{k} \in \mathcal{R}(A^T)$, $u_{k}$ $\in \mathcal{R}(A^T)$, and thus $x_{k}$ $\in \mathcal{R}(A^T)$ provided $x_0 \in \mathcal{R}(A^T)$. 

    Therefore, when $x_k$ is an approximate solution of \autoref{eq:linear_system} obtained through \autoref{algo:fabgmres_dobk}, it is the solution minimizing the norm of residual, since $x_k \in \mathcal{R}(A^T) \perp \mathcal{N}(A)$.
\end{proof}

\section{Numerical experiments }\label{sec:numerical_experiments}

In this section, we present numerical experiments to illustrate the gains of the proposed methods over randomized linear solvers and the Krylov subspace solvers. We describe some parameters and facets of algorithmic implementation followed in our experiments, in the corresponding sections below. In \autoref{sec:results_randomized_solvers}, we compare the proposed Regularized Orthogonality and Residual based Block-Kaczmarz method (ROR-BK) with
\begin{itemize}
\item SOBK: Simple Orthogonal Block Kaczmarz \cite{0}
\item TA-ReBlocK-U: Tail Averaged Regularized Block Kaczmarz (TA-ReBlocK) with uniform sampling \cite{ReRBK}.
\end{itemize}

to demonstrate that utilizing the degree of orthogonality among all blocks, regularization in the block solutions, and a residual based block iterate, can significantly enhance the efficiency. Later in \autoref{sec:results_Krylov_solvers}, we also present experimental results highlighting the gains of the proposed method compared to well known Krylov solvers LSQR, AB/BA-GMRES, and their preconditioned versions. We also present results of practical image reconstruction to highlight its real-world advantages. In our experiments, we use a zero vector as the initial solution \(x_0\). However, we propose an \(x_0 \in \mathcal{R}(A^T)\) that can be efficiently evaluated to satisfy \(\frac{\|b-Ax_0\|_2}{\|b\|_2} \le 1\) as in Algorithm \autoref{algo:initial_solution} presented in the appendix. All experiments were performed on a MacBook Pro equipped with an Apple M4 Pro SoC featuring a 14-core CPU and 20-core GPU, 48 GB unified memory, 512 GB SSD, running macOS Tahoe 26.1 and MATLAB R2025b. The code corresponding to the reported experiments can be found at \cite{Codes}.

\subsection{Comparison with randomized solvers} \label{sec:results_randomized_solvers}

\begin{table}
\centering
\begin{tabular}{llccccc}
\hline
\textbf{Matrix} & & \textbf{Franz10} & \textbf{relat7} & \textbf{EX6} & \textbf{lpl3} & \textbf{nemswrld} \\
\hline
$m \times n$ & & $19588 \times 4164$ & $21924 \times 1045$ & $6545 \times 6545$ & $10828 \times 33686$ & $7138 \times 28550$ \\
Density & & $0.12\%$ & $0.36\%$ & $0.69\%$ & $0.03\%$ & $0.09\%$ \\
Cond$(A)$ & & $1.27\text{e}+16$ & $\infty$ & $5.32\text{e}+58$ & $5.34\text{e}+26$ & $\infty$ \\
\hline
\multirow{2}{*}{RBK$(k)$} & IT & $788$ & $1039$ & $48243$ & $10715$ & $47461$ \\
                          & CPU & $19.2037$ & $7.6635$ & $85.4743$ & $154.6893$ & $343.8807$ \\
\hline
\multirow{2}{*}{GBK} & IT & $156$ & $975$ & $6239$ & $\ddagger$ & $\ddagger$ \\
                     & CPU & $15.8140$ & $150.4485$ & $782.2268$ & $\ddagger$ & $\ddagger$ \\
\hline
\multirow{2}{*}{GRBK} & IT & $608$ & $1413$ & $47978$ & $\ddagger$ & $\ddagger$ \\
                      & CPU & $8.5032$ & $4.5733$ & $1609.3189$ & $\ddagger$ & $\ddagger$ \\
\hline
\multirow{2}{*}{GK} & IT & $19201$ & $46908$ & $869077$ & $2284891$ & $\ddagger$ \\
                    & CPU & $4.5091$ & $7.9804$ & $458.9739$ & $1758.0377$ & $\ddagger$ \\
\hline
\multirow{2}{*}{SOBK} & IT & $453$ & $700$ & $27742$ & $3768$ & $65998$ \\
                      & CPU & $\mathbf{2.5101}$ & $\mathbf{1.3119}$ & $\mathbf{12.5711}$ & $\mathbf{13.7768}$ & $\mathbf{106.1289}$ \\
\hline
\end{tabular}
\caption{Numerical comparisons of SOBK against a few state-of-the-art block-Kaczmarz methods \cite{0}. The stopping criterion is based on the relative residual norm (RRN), defined as RRN=$\frac{\|b-A\tilde{x}\|_2}{\|b\|_2}< 1e-06$ , where $\tilde{x}$ denotes the computed solution. Relative Error (RE)=$\frac{\|\tilde{x}-x_\star\|_2}{\|x_\star\|_2}$, the true solution $x_\star$ being randomly generated with entries from $\mathcal{N}(0,1)$ distribution . `$\ddagger$': The method did not meet the stopping criterion within 2000 seconds. IT - Number of iterations performed; CPU - Elapsed computing time in seconds. RBK($k$): Randomized block Kaczmarz method with $k$-means clustering \cite{14}. GBK: A greedy block Kaczmarz algorithm \cite{15}. GRBK: Greedy randomized block Kaczmarz method \cite{14}. GK: Greedy Kaczmarz method \cite{10}.}
\label{fig:SOBK_experiment_table}
\end{table}

%\subsection{Implementation details} \label{sec:implementation}

Table 1 presented in \cite{0} (and reproduced here) highlights that SOBK is a significant improvement over the prior randomized Kaczmarz methods. In our further results comparing the proposed method to SOBK and TA-ReBlock-U, we fixed the number of blocks in the linear systems as 100 giving the number of rows in a block $k=\lfloor \frac{m}{100}\rfloor$. For TA-ReBlock-U, the regularization parameter $\mu= 1e-3 \times$(number of rows in a block), $T_b=300$ as suggested in \cite{ReRBK,0}, where we average the last $T_b$ number of updates of $x$'s for the final solution of TA-ReBlocK-U. If TA-ReBlock-U converges within 300 iterations, we consider only the recent approximation (i.e. $T_b=1$). For the proposed ROR-BK method we use $\mu=1e-06 \times$(number of rows in a block).

In ROR-BK, for a chosen block $A_\tau \in \mathbb{R}^{k\times n}$ if $k<n$ we compute and use  $(A_\tau A_\tau^{T}+\mu I)^{-1}$ for an update $x_{t+1}=x_{t}+A_\tau^{T}(A_\tau A_\tau^{T}+\mu I)^{-1}(b_\tau-A_\tau x_{t})$ and store it. Note that using Lemma \autoref{thm:woodberry_identity}, $x_{t}+A_\tau^{T}(A_\tau A_\tau^{T}+\mu I)^{-1}(b_\tau-A_\tau x_{t})=x_t+ (A_\tau^{T}A_\tau+\mu I)^{-1}A_\tau^{T}(b_\tau-A_\tau x_t)$. So while $k\ge n$, we compute and store $(A_\tau^{T}A_\tau+\mu I)^{-1}$ for the update $x_{t+1}= x_t+ (A_\tau^{T}A_\tau+\mu I)^{-1}A_\tau^{T}(b_\tau-A_\tau x_t)$.

When $k < n$, the computation cost for updating $x_{t+1}$ from $x_t$ is $\mathcal{O}(k^3 + nk^2)$, where the number of blocks is typically $\mathcal{O}(\sqrt{m})$. If the inversion is pre-computed and stored, each ROR-BK update costs $\mathcal{O}(nk)$. Thus the order of arithmetic operations in each iteration is given by the evaluation of residual at $\mathcal{O}(mn)$, which may be further improved by any implicit methods of evaluation if possible. To ensure a fair comparison, we also store $A_\tau^\dagger$ to avoid any such repeated evaluation in SOBK as well. A similar approach in computation was followed for $n \le k$.

As proposed in \cite{ReRBK}, a Cholesky-based linear solver is used to compute $(A_\tau A_\tau^T +\mu I)^{-1}(b_\tau-A_\tau x_t)$ in TA-ReBlock-U, which is inefficient compared to our implementation of ROR-BK especially when the number of iterations is large, as we avoid the more cumbersome re-evaluation of the inverse.

\autoref{tab:dobk_vs_sobk_vs_rerbk_tab1}, \autoref{tab:dobk_vs_sobk_vs_rerbk_tab2}, \autoref{tab:dobk_vs_sobk_vs_rerbk_tab3}, and \autoref{tab:dobk_vs_sobk_vs_rerbk_tab4} show the experimental results for different types of matrices in the University of Florida sparse matrix collection \cite{16}, and for randomly generated dense matrices of different sizes. For \autoref{tab:dobk_vs_sobk_vs_rerbk_tab3} with well-conditioned matrices, we generated $b=Ax, x_\star = A^\dagger b$, where $x \in \mathbb{R}^{n \times 1}$ was generated using MATLAB `randn' function. For the other tables, we compute \( b = Ax_\star \) by randomly generating \( x_\star \) using \( \mathcal{N}(0,1) \) distribution. For the cases where $A$ is from the University of Florida sparse matrix collection \cite{16}, the experimental results are reported after averaging over 50 different cases of $b$. For the randomized experiments, we generate 5 different matrices $A$, and for each case generate 50 different $x_\star$ with entries from \(\mathcal{N}(0,1)\). The average results of the above are reported for those instances. The stopping criterion for trial problems is based on the relative residual norm RRN$=\frac{\|b-A\tilde{x}\|_2}{\|b\|_2}< 1e-06$, where $\tilde{x}$ denotes the computed solution, unless stated otherwise. The observed relative error (RE)=$\frac{\|x-x_\star\|_2}{\|x_\star\|_2}$ indicates that indeed the solutions are converging even as we minimize the norm of the residual in these highly ill-condition problems.

We use speed-up of ROR-BK over a given method and given tolerance as $\frac{\text{time (or iterations) of method}}{\text{time (or iterations) of ROR-BK}}$ and we write it as `method:Speed-up'. `$\ddagger$' denotes that the method did not meet the convergence criterion within 2000 seconds. We use \enquote{IT} to denote the required number of iterations to meet the convergence criterion, and the elapsed computing time in seconds is referred to as \enquote{CPU}.

\begin{table}
    \centering
    \begin{tabular}{|c|c|c|c|c|}
        \hline
        Matrix & &Franz10 &  n3c6-b7 & lp-pds-02\\
        \hline
        Size & &19588 × 4164 & 6435 × 6435 &2953 × 7716\\
        \hline
        Density & & 0.12\% & 0.12\% &  0.07\%\\
        \hline
        cond(A)& &1.27e+16 &4.99e+202&  6.25e+15 \\
        \hline
        prob-cond& &1.6e+15  &1.78e+19 & 6.18e+15 \\
        \hline
\multirow{3}{*}{SOBK} & IT &540.90  &796.65 &16707.55 \\
                       & CPU &2.69 &2.12 & 13.80\\
                       &RE&3.82e-01 & 6.83e-01& 7.86e-01\\
        \hline
\multirow{3}{*}{TA-ReBlocK-U} & IT &899.90 &1618.90 &37323.55\\
                       & CPU & 5.64 &6.23 &186.30 \\
                        &RE&3.82e-01  &6.83e-01 &7.86e-01\\
        \hline
% \multirow{3}{*}{DOBK} & IT &103 & 160.35&2900.35  \\
%                        & CPU & 1.57 & 0.73&3.10 \\
%                        &RE&3.82e-01  & 6.83e-01& 7.86e-01\\
%         \hline
 \multirow{3}{*}{ROR-BK} & IT &108.40 & 143.30&2429.7  \\
                       & CPU & 1.49 & 0.71&3.94 \\
                       &RE&3.82e-01  & 6.83e-01& 7.87e-01\\
        \hline       
\multirow{2}{*}{SOBK:Speed-up} & IT &\textbf{4.98}  &\textbf{5.55} &\textbf{5.92} \\
               & CPU &\textbf{1.79}&\textbf{2.98}& \textbf{3.50} \\
                      
        \hline
\multirow{2}{*}{TA-ReBlocK-U:Speed-up} & IT & \textbf{8.30} &\textbf{11.29} &\textbf{15.36} \\
                       & CPU &\textbf{3.77}&\textbf{8.77}&\textbf{47.28} \\
                      
        \hline
    \end{tabular}
    \caption{Ill-conditioned linear systems. The stopping criterion is based on the relative residual norm (RRN), defined as RRN=$\frac{\|b-A\tilde{x}\|_2}{\|b\|_2}< 1e-06$ , where $\tilde{x}$ denotes the computed solution. Relative Error (RE)=$\frac{\|\tilde{x}-x_\star\|_2}{\|x_\star\|_2}$, the true solution $x_\star$ being randomly generated with entries from $\mathcal{N}(0,1)$ distribution . `$\ddagger$': The method did not meet the stopping criterion within 2000 seconds. IT - Average number of iterations performed; CPU - Average computing time in seconds. Note that scaling factors between CPU time and iterations (IT) are a function of the block sizes in the particular example matrix.}
    \label{tab:dobk_vs_sobk_vs_rerbk_tab1}
\end{table}

% \begin{table}[H]
%     \centering
%     \begin{tabular}{|c|c|c|c|c|c|}
%         \hline
%         Matrix & &graphics & 192bit & fome11 & $\text{Maragal-}5$\\
%         \hline
%         Size & &29493 × 11822& 13691 × 13682 &12142 × 24460& 4654 × 3320\\
%         \hline
%         Density & & 0.03\% &0.08\% & 0.02\% &0.60\%\\
%         \hline
%         cond(A)& &1.59e+08& 3.46e+16& 3.98e+04& 7.40e+31 \\
%         \hline
%         prob-cond&&2.9e+06&77.84 &37.42 & 1.4e+04 \\
%         \hline

% \multirow{3}{*}{SOBK} & IT & & & &  \\
%                        & CPU &  &  & &  \\
%                        &RE& & & &\\
%         \hline
% \multirow{3}{*}{TA-ReBlocK-U} & IT & &  & &   \\
%                        & CPU &  &  & &  \\
%                        &RE& & & &\\
%         \hline
%     \multirow{3}{*}{DOBK} & IT & &  & &   \\
%                        & CPU &  &  & &  \\
%                        &RE& & & &\\
%         \hline
%         \multirow{2}{*}{SOBK:Speed-up} & IT & & &&  \\
%                        & CPU &&& &  \\
%         \hline
%          \multirow{2}{*}{TA-ReBlocK-U:Speed-up} & IT & & & &  \\
%                        & CPU & & &  &  \\
%         \hline 
%     \end{tabular}
%     \caption{Numerical comparison for a few slightly ill-conditioned systems }
%     \label{tab:dobk_vs_sobk_vs_rerbk_tab3}
% \end{table}

\begin{table}
    \centering
    \begin{tabular}{|c|c|c|c|c|c|}
        \hline
        Matrix & &randn & randn & randn& bcsstm25\\
        \hline
        Size & &60000 × 10000 & 60000 × 20000 &20000 × 60000& 15439 × 15439\\
        \hline
        Density &&100\% &100\% & 100\%& 0.006\%\\
        \hline
        cond(A)&&2.73e+00 &3.72e+00 &3.72e+00 &6.05e+09  \\
        \hline
        prob-cond&&1.68e+00 &2.36e+00 & 1.35e+00 & 5.94e+08  \\
        \hline

\multirow{3}{*}{SOBK} & IT &218.01 &859.80 & 2902.55&249.50 \\
                       & CPU & 16.87 & 176.47 &450.76 & 4.32 \\
                       &RE&1.14e-06&1.60e-06 & 8.16e-01&4.86e-07\\
        \hline
\multirow{3}{*}{TA-ReBlocK-U} & IT &771.10 &2053.15  &6009.70&  1196.65 \\
                       & CPU &40.03  &570.89 &1705.06 & 16.07 \\
                       &RE&1.28e-06 & 1.97e-06&8.16e-01 &1.04e-02\\
        \hline
% \multirow{3}{*}{DOBK} & IT &11 & 241.45 &872.50& 38.15  \\
%                        & CPU & 1.38 &122.44 &335.92 & 0.95 \\
%                        &RE& 7.85e-07& 1.69e-06& 8.16e-01&6.08e-02\\
%         \hline
\multirow{3}{*}{ROR-BK} & IT &63.05 & 215.15 &754.70& 33.4  \\
                       & CPU & 13.61 &48.43 & 87.17&0.91  \\
                       &RE& 1.02e-06& 1.66e-06&8.16e-01 &6.04e-02\\
        \hline

\multirow{2}{*}{SOBK:Speed-up} & IT &\textbf{3.46} & \textbf{3.81}&\textbf{3.84}&\textbf{7.47}  \\
                       & CPU &\textbf{1.23}&\textbf{3.62}& \textbf{5.17}& \textbf{4.72} \\
        \hline
\multirow{2}{*}{TA-ReBlocK-U:Speed-up} & IT &\textbf{12.23} & \textbf{9.11}&\textbf{7.96}&\textbf{35.82}  \\
                       & CPU &\textbf{2.94}&\textbf{11.78}& \textbf{19.56}& \textbf{17.57} \\
        \hline
         
    \end{tabular}
    \caption{Well-conditioned and Moderately ill-conditioned number linear systems. The stopping criterion is based on the relative residual norm (RRN), defined as RRN=$\frac{\|b-A\tilde{x}\|_2}{\|b\|_2}< 1e-06$ , where $\tilde{x}$ denotes the computed solution. Relative Error (RE)=$\frac{\|\tilde{x}-x_\star\|_2}{\|x_\star\|_2}$, the true solution $x_\star$ being randomly generated with entries from $\mathcal{N}(0,1)$ distribution . `$\ddagger$': The method did not meet the stopping criterion within 2000 seconds. IT - Average number of iterations performed; CPU - Average computing time in seconds. Note that scaling factors between CPU time and iterations (IT) are a function of the block sizes in the particular example matrix.}
    \label{tab:dobk_vs_sobk_vs_rerbk_tab2}
\end{table}

\begin{table}
    \centering
    \begin{tabular}{|c|c|c|c|c|c|}
        \hline
        Matrix & &lp-80bau3b & cage10 &  abtaha1\\
        \hline
        Size & &2262 × 12061& 11397×11397  & 14596 × 209\\
        \hline
        Density & & 0.08\% &0.11\% & 1.68\%\\
        \hline
        cond(A)& &5.67e+02& 1.10e+01&  1.22e+01 \\
        \hline
        prob-cond&&1.37e+01&7.34e+00 & 9.84e+00 \\
        \hline

\multirow{3}{*}{SOBK} & IT &3479.05 &1260.70 & 190.01  \\
                       & CPU &5.40 &7.24 & 0.18   \\
                       &RE&9.02e-06 &1.31e-06 &2.10e-06   \\
        \hline
\multirow{3}{*}{TA-ReBlocK-U} & IT &14592.65 &12201.85 & 524.11  \\
                       & CPU &52.54 &160.01 &  0.17  \\
                       &RE&9.12e-06 &6.40e-06 & 6.22e-06  \\
        \hline
% \multirow{3}{*}{DOBK} & IT &721.40 & 287.85& 12.01  \\
%                        & CPU &2.08 & 4.10&  0.26  \\
%                        &RE&6.30e-06 &1.22e-06 & 1.16e-06  \\
%         \hline
\multirow{3}{*}{ROR-BK} & IT &624.15 &270.40 & 18.72  \\
                       & CPU &1.55 & 3.26& 0.09   \\
                       &RE&9.02e-06 &1.38e-06 & 1.56e-06  \\
        \hline

\multirow{2}{*}{SOBK:Speed-up} & IT & \textbf{5.57}&\textbf{4.66} &\textbf{10.15} \\
                       & CPU &\textbf{3.48}&\textbf{2.22}& \textbf{2.07}  \\
        \hline
\multirow{2}{*}{TA-ReBlocK-U:Speed-up} & IT &\textbf{23.95} &\textbf{45.12} & \textbf{27.99}  \\
                       & CPU &\textbf{33.89} & \textbf{49.08}& \textbf{1.89}  \\
        \hline 
    \end{tabular}
    \caption{Well-conditioned linear systems. The stopping criterion is based on the relative residual norm (RRN), defined as RRN=$\frac{\|b-A\tilde{x}\|_2}{\|b\|_2}< 1e-06$ , where $\tilde{x}$ denotes the computed solution. Relative Error (RE)=$\frac{\|\tilde{x}-x_\star\|_2}{\|x_\star\|_2}$, the true solution $x_\star=A^\dagger b$. `$\ddagger$': The method did not meet the stopping criterion within 2000 seconds. IT - Average number of iterations performed; CPU - Average computing time in seconds. Note that scaling factors between CPU time and iterations (IT) are a function of the block sizes in the particular example matrix.}
\label{tab:dobk_vs_sobk_vs_rerbk_tab3}
\end{table}

\begin{table}
    \centering
    \begin{tabular}{|c|c|c|c|c|}
        \hline
        Matrix & &1+rand & 1+rand & 1+rand \\
        \hline
        Size & &100000 × 20000& 100000 × 10000 &100000× 15000\\
        \hline
        Density & & 100\% &100\% & 100\% \\
        \hline
        cond(A)& &1.32e+03& 1.46e+02 & 3.31e+02 \\
        \hline
        prob-cond&&3.79e+00&3.25e+00&4.32e+00 \\
        \hline

\multirow{3}{*}{SOBK} & IT &705.1 &268.50 & 458.30   \\
                       & CPU &419.59 &73.80 & 139.87   \\
                       &RE& 2.61e-06&5.39e-06 & 4.39e-06 \\
        \hline
\multirow{3}{*}{TA-ReBlocK-U} & IT &889.20 &264.80 & 660.40  \\
                       & CPU &161.19 &59.57 & 225.18   \\
                       &RE&2.98e-06 &5.40e-06 & 4.82e-06 \\
        \hline
% \multirow{3}{*}{DOBK} & IT &20.11 &39.25 & 66.80  \\
%                        & CPU &10.77 & 23.20& 50.39    \\
%                        &RE&3.36e-06 &4.60e-06 & 4.38e-06 \\
%         \hline
\multirow{3}{*}{ROR-BK} & IT &79.01 &29.50 & 49.35  \\
                       & CPU &87.73 &16.31 & 35.30   \\
                       &RE&2.47e-06 &4.54e-06 & 4.06e-06 \\
                       \hline

\multirow{2}{*}{SOBK:Speed-up} & IT &\textbf{8.92} & \textbf{9.10}& \textbf{9.28} \\
                       & CPU & \textbf{4.78}&\textbf{4.52} &  \textbf{3.96} \\
        \hline
\multirow{2}{*}{TA-ReBlocK-U:Speed-up} & IT &\textbf{11.25} &\textbf{8.97} & \textbf{13.38}   \\
                       & CPU & \textbf{1.83}&\textbf{3.65} &   \textbf{6.37} \\
        \hline
    \end{tabular}
    \caption{Highly over-determined dense systems with entries from $\mathcal{U}(1,2)$. The stopping criterion is based on the relative residual norm (RRN), defined as RRN=$\frac{\|b-A\tilde{x}\|_2}{\|b\|_2}< 1e-06$ , where $\tilde{x}$ denotes the computed solution. Relative Error (RE)=$\frac{\|\tilde{x}-x_\star\|_2}{\|x_\star\|_2}$, the true solution $x_\star$ being randomly generated with entries from $\mathcal{N}(0,1)$ distribution . `$\ddagger$': The method did not meet the stopping criterion within 2000 seconds. IT - Average number of iterations performed; CPU - Average computing time in seconds. Note that scaling factors between CPU time and iterations (IT) are a function of the block sizes in the particular example matrix.}
    \label{tab:dobk_vs_sobk_vs_rerbk_tab4}
\end{table}

Our experiments demonstrate that with respect to time, ROR-BK is \textbf{1.23} to \textbf{5.17} times faster than SOBK and \textbf{1.83} to \textbf{49.08} times faster than the TA-ReBlocK-U method in these example systems, highlighting the gains of the proposed method across a wide range of problems. While the speed-up of ROR-BK in the number of iterations is even more attractive, we have to note that SOBK has 3 block updates every iteration and the other two methods have 4 block updates in each.

\subsection{Comparison with Krylov solvers and example real-world applications} \label{sec:results_Krylov_solvers}

As a baseline for Krylov solvers, we consider LSQR \cite{lsqr} and AB/BA-GMRES. In both the AB- and BA-GMRES frameworks, we set $B = A^T$. To optimize for the computing efficiency given the ratio of dimensions of the matrix, underdetermined systems were addressed via the AB-GMRES formulation to solve the normal equations of the second kind ($AA^T y = b$), whereas overdetermined systems employed the BA-GMRES formulation to solve the normal equations of the first kind ($A^TA x = A^T b$). Because the satisfactory performance of Krylov solvers on highly ill-conditioned systems requires preconditioning, we evaluated three distinct preconditioning paradigms namely Jacobi, ILU, and RIF preconditioners \cite{11,12,gmres3,rif}. First, we utilized the unconditionally stable matrix-free Jacobi (diagonal equilibration) preconditioner for LSQR (referred as Jacobi-LSQR), and AB/BA-GMRES (referred as Jacobi-GMRES). To establish a robust preconditioning baseline without zero-pivot breakdowns inherent to severely rank-deficient matrices, we employed Tikhonov-regularized incomplete factorizations for ILU-GMRES. This configuration was geometrically adaptive: square systems utilized Incomplete LU with Threshold and Pivoting (ILUTP) with a regularizing micro-shift, while general rectangular systems dynamically deployed Left-Preconditioned AB-GMRES or BA-GMRES based on the ratio of matrix dimensions. Crucially, these rectangular GMRES preconditioners utilized a strictly zero-fill, diagonally shifted ILU(0). Because the explicitly formed shifted normal operators ($AA^T + \alpha I$ or $A^TA + \alpha I$) possess strict diagonal dominance, ILU(0) guarantees stable factorization; ignoring the computationally expensive dynamic pivoting of ILUTP here maximized the computational efficiency without sacrificing stability. Conversely, ILU-LSQR strictly requires an $n \times n$ right-preconditioner acting on the column space. For this solver, the preconditioner was constructed using Incomplete LU with Threshold and Pivoting (ILUTP) applied to $A^T A + \alpha I$. Unlike the GMRES formulation, this dynamic pivoting strategy was explicitly required to stabilize LSQR's bidiagonalization process against the severe ill-conditioning of the unpermuted operator, where standard zero-fill ILU inherently fails.
Finally, we include the Robust Incomplete Factorization (RIF) preconditioner, referred to as RIF-GMRES and RIF-LSQR when used with GMRES and LSQR, respectively. This is used to demonstrate the algorithmic fragility and catastrophic floating-point cancellation that frequently occurs when global approximate inverses amplify the null-space of ill-conditioned normal equations. 
 
The experimental set was divided into two tests of pure algebraic robustness using matrices with sparsity $\leq 50\%$ and different dimensions (\autoref{tab:rorbk_lsqr_gmres}), and the convergence evaluations on real-world tomographic data with measurement noise included as well, using the \texttt{paralleltomo} generator from the AIR Tools II package (\autoref{tab:ct_test}) \cite{air}. The tomographic operators ($A$) were constructed simulating a parallel-beam geometry with full detector coverage ($N_p = \lceil\sqrt{2}N\rceil$), for a true image of size $N \times N$. To stringently test the solvers' implicit regularization capabilities on severely underdetermined systems ($m \ll n$), we simulated extreme sparse-view CT scenarios by restricting the angular sampling.
For the test with \texttt{paralleltomo}, the experiments utilized the standard Shepp-Logan phantom discretized on a $1024 \times 1024$ grid. To simulate real-world sensor inconsistencies, the exact vector was corrupted with scaled Gaussian noise ($10^{-4}$ relative noise level) prior to reconstruction. This piecewise-constant synthetic dataset serves as a strict mathematical baseline to verify basic convergence behavior for a noisy system and algebraic robustness.

While such standard analytical phantoms provide excellent theoretical baselines, they lack the high-frequency gradients inherent to genuine physical anatomy. So to verify that the proposed method does not artificially over-perform on piecewise-constant regions, we experiment on another \emph{empirical} texture dataset named as \texttt{mri}. A high-resolution ($640 \times 640$) empirical human brain slice was extracted from the standard anonymized MRI sample dataset natively provided within the MATLAB environment \cite{matlab}, and mapped directly into the discrete reconstruction space. Crucially, to isolate the algorithm's capability to resolve complex biological gradients from the confounding effects of forward-model discretization error, the target sinogram was generated consistently within the discrete domain ($b=Ax_{\text{true}}$) before the addition of  stochastic noise of relative magnitude $10^{-4}$. By maintaining this discrete algebraic consistency, we guarantee that the resulting Peak Signal-to-Noise Ratio (PSNR) and Structural Similarity Index (SSIM) metrics exclusively reflect the solver's algorithmic resilience to severe data sparsity and noise, uncorrupted by structural model mismatch.

As mentioned, to simulate real-world measurement uncertainties in these discrete datasets, additive Gaussian noise is introduced in the vectors $b$.:
\begin{equation}
    b = Ax_\star + \frac{10^{-4} \|Ax_\star\|_2}{\|v_{noise}\|_2} \, v_{noise}
\end{equation}
where $v_{noise}$ is a standard normal Gaussian vector, and $x_\star$ being the true solution of the system prior to the noise addition. This controlled signal-to-noise ratio allows us to accurately track the true relative error and measure the implicit regularization of the algorithms prior to the onset of noise-fitting divergence.

To evaluate performance, we use a stopping tolerance based on relative error of $1e-2$ for the test cases in \autoref{tab:rorbk_lsqr_gmres} with a maximum iteration limit of 2000, and the practical experiments in \autoref{tab:ct_test} were allowed to exhaust the maximum iterations. When $A$ is fixed, we perform experiments using 25 different realizations of $b$. For experiments with randomly generated matrices $A$, we generate 5 different matrices, and for each matrix we use 25 different realizations of $b$, and the final averaged results are reported in tables. For the experiments as reported in \autoref{tab:rorbk_lsqr_gmres} the true solution $x_{\text{true}}$ is generated with entries drawn from $\mathcal{N}(0,1)$. As mentioned earlier, the computational cost per iteration of ROR-BK is $\mathcal{O}(mn)$, whereas for GMRES it grows with iteration number as $\mathcal{O}\big(mn + \min(l,n)\,ln\big)$,
where $l$ is the iteration count.

\begin{table}
    \centering
    \begin{tabular}{|c|c|c|c|c|c|c|}
        \hline
        Matrix & &Franz10&lp-pds-02  &  randint$_{\{0,1\}}$(20000,2000) & gemat11\\
        \hline
        Size & &19588×4164&2953×7716  & 20000 × 2000 & 4929 × 4929 \\
        \hline
        Density & & 0.12\%&0.07\%  & 50.007\% & 0.13\%\\
        \hline
        cond(A)& &1.27e+16&6.25e+15&  4.32e+01 & 5.95e+07\\
        \hline
        prob-cond & &1.13e+15 &5.62e+14 & 2.63e+00& 8.35e+05\\
        \hline

\multirow{3}{*}{AB/BA-GMRES} & IT & $\ddagger$ & $\ddagger$ & 14.5 & $\ddagger$  \\
                       & CPU &36.91 & 19.90 & 0.062 & 20.71 \\
                       &RE&5.44e-01 & 7.88e-01 & \textbf{6.55e-03} & 1.34e+00 \\
\hline
\multirow{3}{*}{Jacobi-GMRES} & IT & $\ddagger$ & $\ddagger$ & 12.7 & $\ddagger$  \\
                       & CPU &37.47 & 19.92 & 0.065 & 20.44 \\
                       &RE&4.83e-01 & 7.88e-01 & \textbf{6.55e-03} & 1.34e+00 \\
\hline
\multirow{3}{*}{ILU-GMRES} & IT & $\ddagger$ & $\ddagger$ & 4.2 &  $\ddagger$ \\
                       & CPU &25.68 & 16.70 & 1.26 & 21.10 \\
                       &RE& 2.24e+00& 7.88e-01 & \textbf{1.35e-03} & 5.82e-01 \\
\hline
\multirow{3}{*}{RIF-GMRES} & IT & $\ddagger$ & $\ddagger$ & 11.1 & $\ddagger$  \\
                       & CPU &141.40 & 45.75 & 26.55 & 93.78 \\
                       &RE&3.73e+00 & 7.88e-01& \textbf{9.94e-03} & 1.61e+00 \\
\hline
\multirow{3}{*}{LSQR} & IT & $\ddagger$ & $\ddagger$ & 12.33 &$\ddagger$   \\
                       & CPU & 16.09& 4.05 & 0.041 & 4.08 \\
                       &RE&2.38e+12 &  4.81e+10& \textbf{6.24e-03} & 1.94e-01  \\
\hline
\multirow{3}{*}{Jacobi-LSQR} & IT &$\ddagger$  &$\ddagger$  & 9.4 & $\ddagger$  \\
                       & CPU &16.01 & 4.05 & 0.05 & 4.07 \\
                       &RE& 7.48e+12&4.86e+10  & \textbf{6.24e-03} &  5.97e-01\\
\hline
\multirow{3}{*}{ILU-LSQR} & IT & $\ddagger$ &$\ddagger$ & 8.33 & $\ddagger$  \\
                       & CPU &16.13 & 4.06 & 0.047 &  4.02\\
                       &RE& 7.48e+12&4.86e+10 & \textbf{6.24e-03} & 5.97e-01 \\
\hline
\multirow{3}{*}{RIF-LSQR} & IT & $\ddagger$ & $\ddagger$ & $\ddagger$ & $\ddagger$  \\
                       & CPU & 127.63& 129.14 & 32.83 & 78.60 \\
                       &RE&4.27e+12 & 1.30e+11 & 2.27e+00 & 2.11e+01 \\
\hline
\multirow{3}{*}{ROR-BK} & IT & $\ddagger$ & 1723.2 & 4.2 & 1893.2  \\
                       & CPU &17.85 &3.89  & 0.045 & 3.93 \\
                       &RE& 3.90e-01& \textbf{9.87e-03}& \textbf{6.36e-03} & \textbf{9.79e-03} \\
\hline
    \end{tabular}
    \caption{Comparison of ROR-BK with LSQR, AB/BA-GMRES, and their preconditioned versions . The stopping criterion is based on the relative error (RE)=$\frac{\|x - x_\star\|_2}{\|x_\star\|_2} < 1e-2$, where \(x_\star\) denotes the true solution with entries from $\mathcal{N}(0,1)$. \(\ddagger\): The method failed to satisfy the stopping criterion within 2000 iterations. IT - Average number of iterations performed; CPU - Average computing time in seconds. randint$_{\{0,1\}}(m,n):$ $m \times n$ dimensional matrix with entries from $\{0,1\}$ with equal probability. Converged cases highlighted in bold. Note that the condition number of he problems reported does not apply for least-square methods solving the normal equations, where the condition numbers can be much higher.}
    \label{tab:rorbk_lsqr_gmres}
\end{table}

\begin{table}
\centering
\setlength{\tabcolsep}{3pt}

\begin{tabular}{|l|cccc|cccc|}
\hline

& \multicolumn{4}{c|}{\texttt{paralleltomo}}
& \multicolumn{4}{c|}{\texttt{mri}} \\
\hline
Size & \multicolumn{4}{c|}{$1448 \times 1048576$}
& \multicolumn{4}{c|}{$1812 \times 409600$} \\
\hline
Density & \multicolumn{4}{c|}{$0.07\%$}
& \multicolumn{4}{c|}{$0.11\%$ }\\
\hline
cond$(A)$ & \multicolumn{4}{c|}{$\infty$} 
& \multicolumn{4}{c|}{$\infty$} \\
\hline
prob-cond & \multicolumn{4}{c|}{$\infty$} 
& \multicolumn{4}{c|}{$\infty$} \\
\hline
Method
&  CPU & RE & PSNR(dB) & SSIM
&  CPU & RE & PSNR(dB) & SSIM \\
\hline

AB/BA-GMRES
&  25.53 & 5.71e+01 & -22.85 & 2.98e-01
&  15.66 & 1.29e+00 & 6.05 & 6.09e-03 \\
\hline

Jacobi-GMRES
&  25.65 & 2.77e+02 & -35.71 & 1.21e-03
&  15.56 & 2.09e+01 & -17.67 & 1.63e-04 \\
\hline

ILU-GMRES
&  22.01 & 6.41e+00 & -19.18 & 2.39e-01
&  14.33 & 1.49e+01 &-15.95  & 4.28e-05 \\
\hline

RIF-GMRES
&  27.25 & 8.21e-01 & 13.84 & 4.06e-01
&  16.30 & 7.09e+09 & -188.78 & -4.08e-21 \\
\hline

LSQR
&  13.74 & 7.92e-01 & 14.17 & 5.95e-01
&  6.58 & 5.56e-01 & 13.32 & 5.93e-01 \\
\hline

Jacobi-LSQR
&  13.48  & 7.92e-01 & 14.17 & 5.95e-01
&  6.71 & 5.56e-01 & 13.32 & 5.93e-01 \\
\hline

ILU-LSQR
&  57.09  &3.70e+02  & -39.23  & 3.12e-01
&  21.53 & 5.52e+02 &-46.60  & 3.36e-01 \\
\hline

SOBK
& 51.18 & 7.92e-01 & 14.17 & 5.95e-01
&  32.97 & 5.56e-01 & 13.32 & 5.93e-01 \\
\hline

TA-ReBlocK-U
& 128.48 & 7.92e-01 & 14.17 &5.95e-01
& 55.01 &  5.56e-01& 13.32 & 5.93e-01 \\
\hline

ROR-BK
&  15.96 & 6.12e-01 & 17.12 &6.65e-01
& 8.03 & 3.26e-01 & 17.32 & 6.41e-01 \\
\hline

\end{tabular}

\caption{
Image reconstruction: Comparison of ROR-BK with LSQR, AB/BA-GMRES, and their preconditioned versions and the block-Karcmarz methods SOBK, TA-ReBlocK-U. For severely underdetermined systems ($m << n$), directly applying RIF to right-precondition LSQR requires the construction of an $n \times n$ approximate inverse of $A^T A$, which is computationally prohibitive, and hence RIF-LSQR is not included in comparisons. All methods were stopped at 2000 iterations; CPU - Average computing time in seconds. In the above example, the matrices have a relatively smaller number of rows and we reduce the number of blocks to 10 for SOBK, TA-ReBlocK-U, and ROR-BK.}
\label{tab:ct_test}
\end{table}

\begin{figure*}
\centering

% Row 1
\begin{minipage}{0.16\textwidth}
\centering
\includegraphics[width=\linewidth]{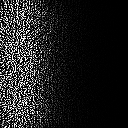}\\
{\small ILU-LSQR}
\end{minipage}
\hfill
\begin{minipage}{0.16\textwidth}
\centering
\includegraphics[width=\linewidth]{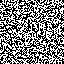}\\
{\small Jacobi-GMRES}
\end{minipage}
\hfill
\begin{minipage}{0.16\textwidth}
\centering
\includegraphics[width=\linewidth]{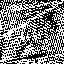}\\
{\small AB/BA-GMRES}
\end{minipage}
\hfill
\begin{minipage}{0.16\textwidth}
\centering
\includegraphics[width=\linewidth]{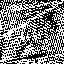}\\
{\small ILU-GMRES}
\end{minipage}
\hfill
\begin{minipage}{0.16\textwidth}
\centering
\includegraphics[width=\linewidth]{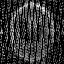}\\
{\small RIF-GMRES}
\end{minipage}
\hfill
\begin{minipage}{0.16\textwidth}
\centering
\includegraphics[width=\linewidth]{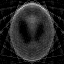}\\
{\small LSQR}
\end{minipage}

\vspace{2mm}

% Row 2
\begin{minipage}{0.18\textwidth}
\centering
\includegraphics[width=\linewidth]{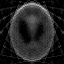}\\
{\small Jacobi-LSQR}
\end{minipage}
\hfill
\begin{minipage}{0.18\textwidth}
\centering
\includegraphics[width=\linewidth]{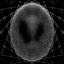}\\
{\small SOBK}
\end{minipage}
\hfill
\begin{minipage}{0.18\textwidth}
\centering
\includegraphics[width=\linewidth]{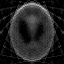}\\
{\small TA-ReBlocK-U}
\end{minipage}
\hfill
\begin{minipage}{0.18\textwidth}
\centering
\includegraphics[width=\linewidth]{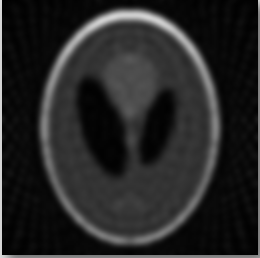}\\
{\small ROR-BK}
\end{minipage}
\hfill
\begin{minipage}{0.18\textwidth}
\centering
\includegraphics[width=\linewidth]{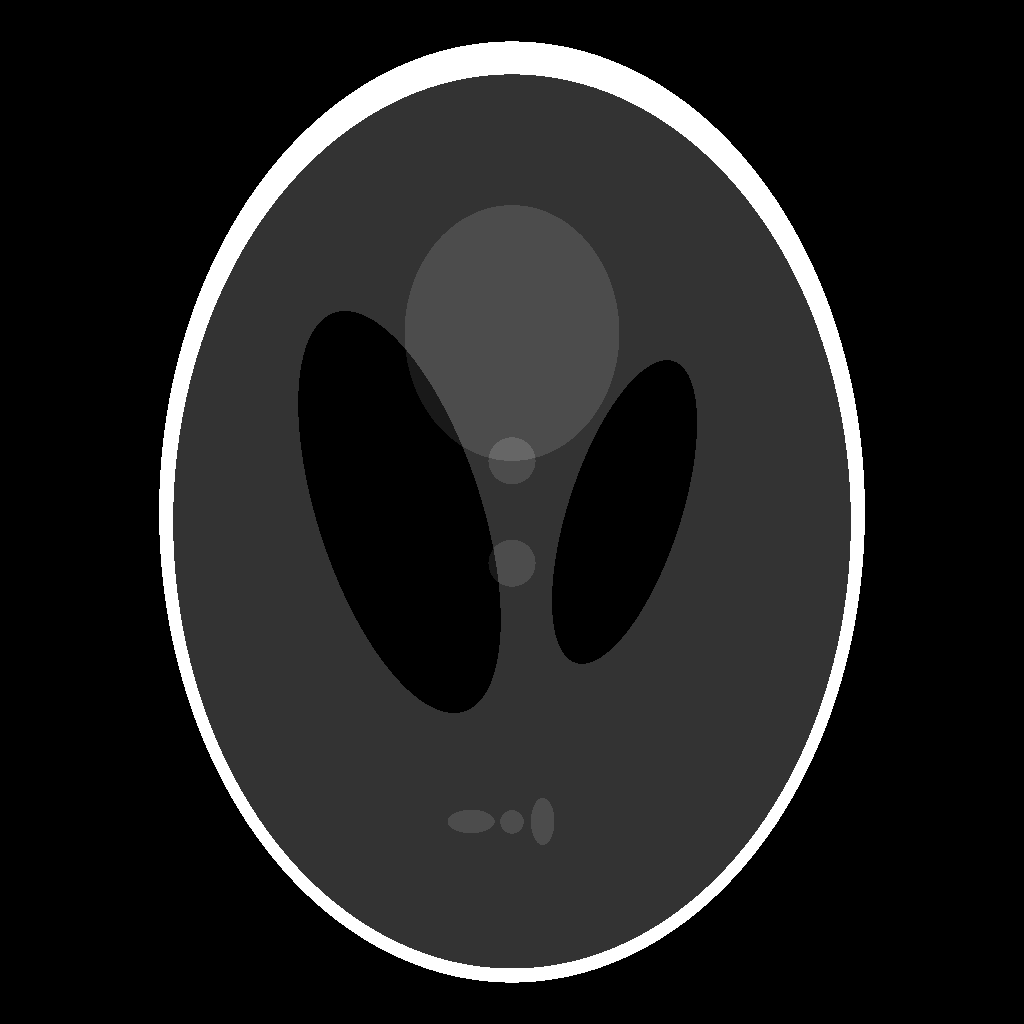}\\
{\small True Image}
\end{minipage}
\vspace{2mm}

% Row 3

\caption{Reconstruction of \texttt{paralleltomo} with an image resolution of $128 \times 128$}
\label{fig:paralleltomo_image}
\end{figure*}

\begin{figure*}
\centering

% Row 1
\begin{minipage}{0.16\textwidth}
\centering
\includegraphics[width=\linewidth]{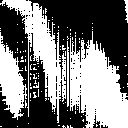}\\
{\small RIF-GMRES}
\end{minipage}
\hfill
\begin{minipage}{0.16\textwidth}
\centering
\includegraphics[width=\linewidth]{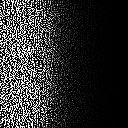}\\
{\small ILU-LSQR}
\end{minipage}
\hfill
\begin{minipage}{0.16\textwidth}
\centering
\includegraphics[width=\linewidth]{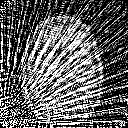}\\
{\small Jacobi-GMRES}
\end{minipage}
\hfill
\begin{minipage}{0.16\textwidth}
\centering
\includegraphics[width=\linewidth]{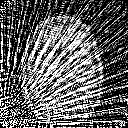}\\
{\small ILU-GMRES}
\end{minipage}
\hfill
\begin{minipage}{0.16\textwidth}
\centering
\includegraphics[width=\linewidth]{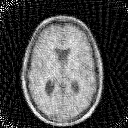}\\
{\small AB/BA-GMRES}
\end{minipage}
\hfill\begin{minipage}{0.16\textwidth}
\centering
\includegraphics[width=\linewidth]{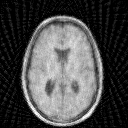}\\
{\small LSQR}
\end{minipage}

\vspace{2mm}

% Row 2
\begin{minipage}{0.18\textwidth}
\centering
\includegraphics[width=\linewidth]{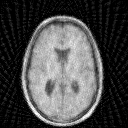}\\
{\small Jacobi-LSQR}
\end{minipage}
\hfill
\begin{minipage}{0.18\textwidth}
\centering
\includegraphics[width=\linewidth]{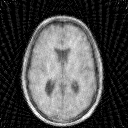}\\
{\small SOBK}
\end{minipage}
\hfill
\begin{minipage}{0.18\textwidth}
\centering
\includegraphics[width=\linewidth]{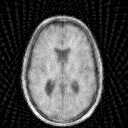}\\
{\small TA-ReBlocK-U}
\end{minipage}
\hfill
\begin{minipage}{0.18\textwidth}
\centering
\includegraphics[width=\linewidth]{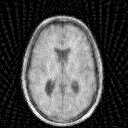}\\
{\small ROR-BK}
\end{minipage}
\hfill
\begin{minipage}{0.18\textwidth}
\centering
\includegraphics[width=\linewidth]{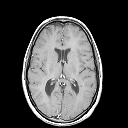}\\
{\small True Image}
\end{minipage}
\vspace{2mm}

% Row 3

\caption{Reconstruction of \texttt{mri} with an image resolution of $128 \times 128$}
\label{fig:mri_image}
\end{figure*}

\section{Conclusion}\label{sec:conclusion}

In this paper, we introduced a block-Kaczmarz algorithm that leverages the concept of orthogonality of a block with all the other blocks, regularization in each iteration for better stability, and sampling of rows based on the current residual into a dynamic block in each iteration, to solve highly ill-conditioned linear systems without preconditioning. This technique is suitable for a wide range of ill-conditioned problems, including square, underdetermined ($m < n$), and overdetermined ($m > n$) cases, and it outperforms recently developed block-Kaczmarz algorithms and the well known Krylov subspace solvers. We also provide a convergence analysis of the improvements in the proposed method. It provides notable gains over the other known methods in sparse systems where effective orthogonality of the blocks is high. It also provides such large gains for highly ill-conditioned dense systems with skewed dimensions, where the regularization of block solutions mitigates possible numerical rank deficiency. The updates to solution based on the residual based dynamic blocks also improve the rate of convergence. In addition to its use as the primary solver, it can also be introduced as a \emph{pre-solver} for widely used iterative methods like CG, GMRES for ill-conditioned problems if needed. Furthermore, by employing this approach as a preconditioner in the inner iteration, the FABGMRES method for solving consistent linear systems may be significantly improved.

\clearpage
%\newpage
\noindent \textbf{Author declarations}:

\vspace{3mm}
\noindent \textbf{Funding}: Murugesan Venkatapathi acknowledges the support of the Science and Engineering Research Board (SERB) grant CRG/2022/004178 in performing this research.

\vspace{3mm}
\noindent \textbf{Conflicts of interest}: The authors do not have any competing financial or non-financial interests to declare.

\bibliographystyle{siamplain}
\bibliography{reference1}

%\clearpage
\noindent \section{Appendix}\label{sec:appendix}

\subsection{Initial Solutions}

Let us now derive a computationally efficient initial solution $x_0 \in \mathcal{R}(A^T)$, such that $\frac{\|b-Ax_0\|}{\|b\|} \le 1$.

\begin{algorithm}
\caption{Initial Solution}\label{algo:initial_solution}
\begin{algorithmic}[1]
\Require $A,b$
\Ensure $x_0$
\State Compute $y \leftarrow$ sum of a set of rows in $A$.
\State $\tilde{b}\leftarrow Ay$
%\State $b \leftarrow \frac{\langle b,\tilde{b}\rangle}{\|\tilde{b}\|_2^2}\tilde{b}+(b-\frac{\langle b,\tilde{b}\rangle}{\|\tilde{b}\|_2^2}\tilde{b})$
\State $x_0 \leftarrow \frac{\langle b, \tilde{b} \rangle}{\|\tilde{b}\|_2^2}y$
\State  Return $x_0$
\end{algorithmic}
\end{algorithm}

\begin{Remark}
\label{thm:initial_solution}
    The initial solution $x_0$ from Algorithm \autoref{algo:initial_solution} is in $\mathcal{R}(A^{T})$, and $\frac{\|b-Ax_0\|}{\|b\|} \le 1$.
    
\end{Remark}

\begin{proof}
We construct the vector $y$ as the sum of a set of rows in A. Thus, $y \in \mathcal{R}(A^{T})$ and we have $\tilde{b}=Ay$. Now, $b$ can be decomposed into two components, one in the space of $\tilde{b}$, and another orthogonal to $\tilde{b}.$
\begin{align}
    b=\frac{\langle b, \tilde{b}\rangle}{\|\tilde{b}\|_2^2}\tilde{b}+(b-\frac{\langle b, \tilde{b}\rangle}{\|\tilde{b}\|_2^2}\tilde{b}) \label{eq:decomp_b}
\end{align}
    
The second term in \autoref{eq:decomp_b} is orthogonal to $\tilde{b}$. Given $\|b\|_2^2=\|\frac{\langle b, \tilde{b}\rangle}{\|\tilde{b}\|_2^2}\tilde{b}\|_2^2+\|(b-\frac{\langle b, \tilde{b}\rangle}{\|\tilde{b}\|_2^2}\tilde{b})\|_2^2$,
we have $\frac{\|b-Ax_0\|_2^2}{\|b\|_2^2} = \frac{\|(b-\frac{\langle b, \tilde{b}\rangle}{\|\tilde{b}\|_2^2}\tilde{b})\|_2^2}{\|b\|_2^2} \le 1$. Also, as $y \in \mathcal{R}(A^{T})$, $x_0=\frac{\langle b, \tilde{b}\rangle}{\|\tilde{b}\|_2^2}y \in \mathcal{R}(A^{T}).$
\end{proof}

\textbf{Note:} The number of arithmetic operations required to evaluate this initial solution $x_0$ through Algorithm \autoref{algo:initial_solution} is $O(mn)$, and note that $y$ can be randomly constructed such that $\tilde{b} \neq 0$, and $\tilde{b}$ is not orthogonal to $b$.

\subsection{Convergence for a weighted least squares problem for the blocks}
Let us denote the sampling procedure of \autoref{algo:roar-bk} as $Q$. One can show that the expected convergence of such a method for a weighted least-square solution with a matrix $W$ representing the weights of the blocks. Let, $S$ be the set of row-indices to form a block $A_S$ from A. Denote $M(A_S)=(A_SA_S^T+\mu I)^{-1}, W(S)=I_S^TM(A_S)I_S, \quad \text{and} \quad P(S)=A_S^TM(A_S)A_S$, $I_S$ being the sub-matrix of $m \times m$ identity matrix with the rows of indices from $S$. Let 
\begin{equation}\label{eq:w_p}
  \bar{W}=\mathbb{E}_{S \sim Q}[W(S)], \text{and} \quad \bar{P}=  \mathbb{E}_{S\sim Q}[P(S)]
\end{equation}
Let's define a weighted minimum norm solution and weighted residual as follows.
\begin{equation}\label{define_xr}
x^{(Q)}=\text{arg}\text{min}_{x\in \mathbb{R}^{n\times 1}} \|Ax-b\|_{\bar{W}}^2, \text {and} \quad r^{(\mu)}=b-Ax^{(Q)}
\end{equation}

\begin{theorem}\label{thm:dobk_for_least_square_problems}
    \textit{Consider the ROR-BK algorithm, namely \autoref{algo:roar-bk} with $M(A_S) = (A_S A_S^\top + \mu I)^{-1}$ and $Q$ be the sampling rule defined earlier. Let $\alpha = \sigma^+_{\min}(\bar{P})$ and assume $x_0 \in \mathcal{R}(A^T)$. Then the expectation of the ROR-BK iterates $x_T$ converges to $x^{(\mu)}$ as}
\begin{equation}
\left\lVert \mathbb{E}[x_T] - x^{(Q)} \right\rVert 
\leq (1 - \alpha)^T \left\lVert x_0 - x^{(Q)} \right\rVert.
\end{equation}
$\sigma^+_{\min}(\bar{P})$ being the minimum non-zero singular value of $\bar{P}.$
\end{theorem}

\begin{proof}
    The proof follows from the Theorem 4.1 in \cite{ReRBK}. Appendix at \cite{ReRBK} outlines the analysis of the convergence of $x^{(Q)}$ to $x_\star$ as $m \rightarrow \infty$.
\end{proof}

%\subsection{Other relevant algorithms}
\begin{algorithm}
\caption{Flexible AB-GMRES with ROR-BK as a preconditioner}\label{algo:fabgmres_dobk}
\begin{algorithmic}[1]
\State Perform uniform aggregation of rows of $A$ and $b$ in sequential order, and obtain $k$ blocks $A_i$ and $b_i$, $i=1,2,\ldots,k$
\State Compute the cosine $C(i,j)$ of the angle between blocks $A_i$ and $A_j$ by the centroid coordinates and the corresponding symmetric matrix $C$.
\State Using entries of $C$, construct probability distribution $\mathcal{P}$.
\State For initial solution $x_0$ compute $r_0 = b - Ax_0$.
\State $\beta = \|r_0\|_2$, $v_1 = r_0/\beta$
\For{$k = 1,2,\ldots$ until convergence}
    \State Apply $\ell_{\max}$ iterations of ROR-BK method to $Az = v_k$ to obtain $z_k = B^{(\ell)} v_k$, where $\ell_{\max}$ is the maximum number of inner iterations allowed for a relative error tolerance $\eta$.
    \[
    \|v_k - AB^{(\ell)}v_k\|_2 \leq \eta \|v_k\|_2.
    \]
    \State $w_k = Az_k$
    \For{$i = 1,2,\ldots,k$}
        \State $h_{i,k} = w_i^T v_k$, \quad $w_k = w_k - h_{i,k}v_i$
    \EndFor
    \State $h_{k+1,k} = \|w_k\|_2$, \quad $v_{k+1} = w_k / h_{k+1,k}$
\EndFor
\State $y_k = \arg\min_{y \in \mathbb{R}^k} \|\beta e_1 - \tilde{H}_k y\|_2$, \quad $u_k = [z_1, z_2, \ldots, z_k]y_k$,\\ where $\tilde{H}_k = \{h_{i,j}\}_{i \in [k+1], j \in [k]}$
\State $x_k = x_0 + u_k$
\end{algorithmic}
\end{algorithm}

\end{document}